\documentclass[a4paper,11pt,reqno]{amsart}

\usepackage{enumerate}
\usepackage[PS]{diagrams}

\newtheorem{theorem}{Theorem}[section]
\newtheorem{corollary}[theorem]{Corollary}
\newtheorem{proposition}[theorem]{Proposition}
\newtheorem{lemma}[theorem]{Lemma}

\theoremstyle{definition}
\newtheorem{definition}[theorem]{Definition}
\newtheorem{remark}[theorem]{Remark}

\let\tensor=\otimes
\let\nbd=\nobreakdash
\let\iso=\cong

\newcommand{\id}{\mathrm{id}}
\newcommand{\bZ}{\ensuremath{\mathbb{Z}}}

\newcommand{\ie}{{\it i.e.}}
\newcommand{\totlt}{\ensuremath{{}{^\mathrm{lt}}\mathrm{Tot}\,}}
\newcommand{\totrt}{\ensuremath{\mathrm{Tot}^\mathrm{rt}\,}}

\newcommand{\prodlt}{\ensuremath{\sideset{^\textrm{lt}}{}\prod}}
\newcommand{\prodrt}{\ensuremath{\sideset{}{^\textrm{rt}}\prod}}

\begin{document}

\title[Double complexes and vanishing of Novikov cohomology]{Double
  complexes and vanishing\\ of Novikov cohomology}

\date{\today}

\author{Thomas H\"uttemann}

\address{Thomas H\"uttemann\\ Queen's University Belfast\\ School of
  Mathematics and Physics\\ Pure Mathematics Research Centre\\ Belfast
  BT7~1NN\\ Northern Ireland\\ UK}

\email{t.huettemann@qub.ac.uk}

\urladdr{http://huettemann.zzl.org/}


\thanks{This work was supported by the Engineering and Physical
  Sciences Research Council [grant number EP/H018743/1].}

\begin{abstract}
  We consider non-standard totalisation functors for double complexes,
  involving left or right truncated products. We show how properties
  of these imply that the algebraic mapping torus of a self map~$h$ of
  a cochain complex of finitely presented modules has trivial negative
  \textsc{Novikov} cohomology, and has trivial positive
  \textsc{Novikov} cohomology provided $h$ is a quasi-isomorphism. As
  an application we obtain a new and transparent proof that a finitely
  dominated cochain complex over a \textsc{Laurent} polynomial ring
  has trivial (positive and nnegative) \textsc{Novikov} cohomology.
\end{abstract}

\maketitle

Finiteness conditions for chain complexes of modules play an important
role in both algebra and topology. For example, given a group~$G$ one
might ask whether the trivial $G$\nbd-module~$\bZ$ admits a resolution
by finitely generated projective $\bZ[G]$\nbd-modules; existence of
such resolutions is relevant for the study of group cohomology of~$G$,
and has applications in the theory of duality groups \cite{Brown}. For
topologists, finite domination of chain complexes is related, among
other things, to questions about finiteness of $CW$ complexes, the
topology of ends of manifolds, and obstructions for the existence of
non-singular closed $1$\nbd-forms \cite{Ranicki-findom, Schuetz}.

A cochain complex~$C$ of $R[z,z^{-1}]$\nbd-modules is called {\it
  finitely dominated} if it is homotopy equivalent, as a complex of
$R$\nbd-modules, to a bounded complex of finitely generated projective
$R$\nbd-modules. Finite domination of~$C$ can be characterised in
various ways; \textsc{Brown} considered compatibility of the functors
$M \mapsto H_* (C;M)$ and $M \mapsto H^* (C;M)$ with products and
direct limits, respectively \cite[Theorem~1]{Brown}, while
\textsc{Ranicki} showed that $C$ is finitely dominated if and only if
the \textsc{Novikov} cohomology of~$C$ is trivial
\cite[Theorem~2]{Ranicki-findom} (see also
Definition~\ref{def:novikov_cohomology} and
Corollary~\ref{cor:ranicki} below).

Our approach to \textsc{Novikov} cohomology is elementary, and
involves a non-standard totalisation functor for double
complexes. Rewriting mapping tori as total complexes of suitable
double complexes, cf.~Remark~\ref{rem:T_by_double} below, we prove a
vanishing result for \textsc{Novikov} cohomology
(Theorem~\ref{thm:torus_cohomology}). As an application we obtain a
new proof of \textsc{Ranicki}'s necessary criterion for finite
domination over \textsc{Laurent} polynomial rings in one variable
(Corollary~\ref{cor:ranicki}). --- The case of several indeterminates
is discussed in papers by \textsc{Sch\"utz} \cite{Schuetz}, and by
\textsc{H\"uttemann} and \textsc{Quinn} \cite{Iteration}.

\section{Truncated product totalisation of double complexes}
\label{sec:realisation}

Let $R$ be a ring with unit. A {\it double complex\/} $D^{*,*}$ is a
$\bZ \times \bZ$\nbd-indexed collection $\big(D^{p,q}\big)_{p,q \in
  \bZ}$ of right $R$\nbd-modules together with ``horizontal'' and
``vertical'' differentials
\[d^h \colon D^{p,q} \rTo D^{p+1,q} \quad \text{and}
\quad d^v \colon D^{p,q} \rTo D^{p,q+1}\] which satisfy the conditions
\[d^h \circ d^h = 0 \ , \quad d^v \circ d^v = 0 \ , \quad d^h \circ
d^v = -d^v \circ d^h \ .\] Note that the differentials anti-commute.
A ``horizontal'' cochain complex in the category of ``vertical''
cochain complexes of right $R$\nbd-modules can be converted to a
double complex in this sense by changing the differential of the $p$th
column by the sign $(-1)^p$. --- We will in general consider unbounded
double complexes so that $D^{p,q} \neq 0$ may occur for $|p|$ and
$|q|$ arbitrarily large.

There are two standard ways to convert a double complex into a cochain
complex via ``totalisation'', one involving direct sums, and one
involving direct products. The former results in a cochain complex
$\mathrm{Tot}_\oplus D^{*,*}$ given by
\[\big( \mathrm{Tot}_\oplus D^{*,*} \big)^n = \bigoplus_{p \in \bZ}
D^{p, n-p}\] with coboundary $d = d^h + d^v$, the latter is defined
analogously with ``$\bigoplus$'' above replaced by ``$\prod$''.

In this paper, which was partially inspired by a preprint of
\textsc{Bergman} \cite[\S6]{Bergman}, we will consider two
non-standard totalisation functors formed by using truncated
products. Given a $\bZ$\nbd-indexed family of modules $M_i$ we define
the {\it left truncated product\/} to be the module
\[\prodlt_i M_i = \bigoplus_{i<0} M_i \,\oplus\, \prod_{i \geq 0} M_i \ ;\]
the elements of this truncated product are ``sequences'' $(m_i)_{i \in
  \bZ}$ with $m_i \in M_i$ such that $m_i = 0$ for $i \ll 0$, which we
might also write in the form $(m_i)_{i \geq k}$ or even $\sum_{i \geq
  k} m_i z^i$ with $z$ being an indeterminate. The latter notation
suggests thinking of such a sequence as a formal \textsc{Laurent}
series with coefficients in the modules~$M_i$. For emphasis and ease
of notation we introduce special notation for the case that all
the~$M_i$ are the same module~$M$; we let $M((z))$ denote the module
of formal \textsc{Laurent} series with coefficients in~$M$,
\[M((z)) = \prodlt M = \Big\{ \sum_{i \geq k} m_i z^i \,|\, k \in
\bZ,\ m_i \in M \Big\} \ .\] Dually we define the {\it right truncated
  product\/} to be the module
\[\prodrt_i M_i = \prod_{i \leq 0} M_i \,\oplus\, \bigoplus_{i>0}
M_i\] of formal \textsc{Laurent} series which are finite to the
right, and define $M((z^{-1}))$ by setting
\[M((z^{-1})) = \prodrt M = \Big\{ \sum_{i \leq k} m_i z^i \,|\, k \in
\bZ,\ m_i \in M \Big\} \ .\]

Note that $R((z))$ and $R((z^{-1}))$ are rings of formal
\textsc{Laurent} series, also known as \textsc{Novikov} rings; there
is a natural identification
\[R((z)) = R[[z]][z^{-1}] \quad \text{and} \quad R((z^{-1})) =
R[[z^{-1}]][z] \ .\] The module $M((z))$ has the structure of an
$R((z))$\nbd-module given by multiplication of formal \textsc{Laurent}
series. Similarly, $M((z^{-1}))$ can be equipped with an obvious
$R((z^{-1}))$\nbd-module structure.

\begin{definition}
  Let $D^{*,*}$ be a double complex. We define its {\it left truncated
    totalisation\/} to be the cochain complex $\totlt D^{*,*}$ which in
  cochain level~$n$ is given by the left truncated product
  \[\big( \totlt D^{*,*} \big)^n = \prodlt_p D^{p, n-p} \ ;\] the
  differential is given by $d = d^h + d^v$. --- Dually, we define the
  {\it right truncated totalisation\/} to be the cochain complex
  $\totrt D^{*,*}$ which in chain level~$n$ is given by the right
  truncated product
  \[\big( \totrt D^{*,*} \big)^n = \prodrt_p D^{p,n-p}\] with
  differential induced as above.
\end{definition}

\begin{proposition}[{\cite[Corollary~29]{Bergman}}]
  \label{prop:main}
  Suppose the double complex $D^{*,*}$ has exact columns. Then $\totlt
  D^{*,*}$ is acyclic. Dually, if $D^{*,*}$ has exact rows then $\totrt
  D^{*,*}$ is acyclic.
\end{proposition}

\begin{proof}
  We prove the first statement only. Abbreviate $\totlt D^{*,*}$
  by~$C$. Suppose $x \in C^n$ is a cocycle. We can write $x = (x_i)_{i \geq
    k}$ with $x_i \in D^{i, n-i}$, and setting $x_j = 0$ for $j<k$ the
  condition $d(x) = 0$ translates into
  \begin{equation}
    \label{eq:middle}
    d^v(x_i) + d^h (x_{i-1}) = 0 \quad \text{for } i \geq k \ .
  \end{equation}
  Set $y_j = 0$ for $j < k$. Suppose by induction on~$i$,
  starting with $i=k$, that we have constructed $y_j \in D^{j,
    n-j-1}$ for $j<i$ such that
  \begin{equation}
    \label{eq:y_j_x_j}
    d^v (y_{j-1}) + d^h (y_{j-2}) = x_{j-1} \quad \text{for } j \leq i \ .
  \end{equation}
  This implies that
  \begin{align*}
    d^v \big( x_i - d^h (y_{i-1})\big) &= \big(
    d^v (x_i) - d^v d^h (y_{i-1}) \big) \\
    &= \big( d^v (x_i) + d^h d^v (y_{i-1}) \big) \\
    &= \big( d^v (x_i) + d^h (x_{i-1} - d^h (y_{i-2})) \big) &&
    \text{(by \eqref{eq:y_j_x_j})} \\
    &= d^v (x_i) + d^h (x_{i-1}) \\
    &= 0 && \text{(by~\eqref{eq:middle})}
  \end{align*}
  so that, by exactness of columns, there exists $y_i \in D^{i,
    n-i-1}$ with $d^v (y_i) = \big( x_i - d^h (y_{i-1})\big)$  or,
  equivalently, $d^v (y_i) + d^h (y_{i-1}) = x_i$.

  This completes the inductive construction. It remains to observe
  that relation~\eqref{eq:y_j_x_j} is now satisfied for all $j \in
  \bZ$ which precisely means that the element $(y_i)_{i \geq k} \in
  C^{n-1}$ is mapped to~$x$ under the coboundary map
  of~$C$. Consequently, $x$~represents the trivial cohomology class
  in~$H^n(C)$ so that $H^n(C) = 0$.
\end{proof}

\begin{remark}
  \label{rmk:not-acyclic}
  The Proposition does not hold for the totalisation functor
  $\mathrm{Tot}_\oplus$ in place of~$\totlt$ or~$\totrt$. For example,
  let $D^{*,*}$ be the double complex defined by setting $D^{p,-p} =
  D^{p,-p-1} = \bZ$ and all other entries~$0$; the horizontal and
  vertical differentials are given by~$-\id_\bZ$ and multiplication
  by~$2$ where possible, respectively. This double complex has exact
  rows, but the element $1 \in D^{0,0} \subset \big(
  \mathrm{Tot}_\oplus D^{*,*} \big)^0$ is a cocycle representing a
  non-zero cohomology class in $H^0\, \mathrm{Tot}_\oplus
  (D^{*,*})$. The same element represents a non-trivial cohomology
  class in $H^0\, \totlt (D^{*,*})$ as well.
\end{remark}

\section{Novikov cohomology of algebraic mapping tori}

\begin{lemma}
  \label{lem:series}
  Suppose that $M$ is a finitely presented right $R$\nbd-module. There
  is a natural $R((z))$\nbd-linear isomorphism
  \[\Phi_M \colon M \tensor_R R((z)) \rTo^\iso M((z)) \ , \quad m
  \tensor \sum_{i \geq k} r_i z^i \mapsto \sum_{i \geq k} mr_i z^i \
  , \] and a similar isomorphism $\Psi_M \colon M \tensor_R
  R((z^{-1})) \rTo^\iso M((z^{-1}))$.
\end{lemma}

\begin{proof}
  We give the proof for~$\Phi_M$ only. First suppose that $F$ is a
  free module with basis $e_1,\, e_2,\, \cdots,\, e_t$. Then every $x
  \in F \tensor_R R((z))$ can be written uniquely in the form $x =
  \sum_{j=1}^t e_j \tensor f_j$ with $f_j \in R((z))$. There exist
  elements $k \in \bZ$ and $r_{ij} \in R$ with $f_j = \sum_{i \geq k}
  r_{ij} z^i$, and~$\Phi_F$ is given by setting
  \[\Phi_F(x) = \sum_{i \geq k} \Big( \sum_{j=1}^t e_j r_{ij} \Big) z^i
  \ .\] This is a well-defined $R$\nbd-module homomorphism. It is
  injective since the $e_j$ form a basis of~$F$; in detail, $\Phi_F(x)
  = 0$ means that $\sum_j r_{ij} e_j = 0$ for all~$i$ so that, by
  linear independence of the~$e_j$, we have $r_{ij} = 0$ for all~$i$
  and~$j$. But this means $x=0$. --- To prove surjectivity, let $g =
  \sum_{i \geq k} m_i z^i \in F((z))$ be given. Since the $e_j$
  generate~$F$ there are elements $r_{ij} \in R$ with $m_i = \sum_j
  e_j r_{ij}$. Set $f_j = \sum_{i \geq k} r_{ij} z^i$ and $x = \sum_j
  (e_j \tensor f_j)$. Then $\Phi_F(x) = y$, by construction.

  For the general case, choose a presentation $G \rTo F \rTo M \rTo 0$
  of~$M$, with $F$ and~$G$ both finitely generated free. The functor
  $N \mapsto N((z))$ is certainly exact (for a map~$f$ we let $f((z))$
  denote the map~$f$ applied componentwise), so we obtain a
  commutative diagram with exact rows
  \begin{diagram}[small,LaTeXeqno]
    \label{diag:chase}
    && G \tensor_R R((z)) & \rTo & F \tensor_R R((z)) & \rTo & M
    \tensor_R R((z)) & \rTo & 0 \\
    && \dTo<{\Phi_G}>\iso && \dTo<{\Phi_F}>\iso && \dDashto<{\Phi_M} \\
    0 & \rTo & G((z)) & \rTo & F((z)) & \rTo & M((z)) & \rTo & 0
  \end{diagram}
  where the dashed arrow is~$\Phi_M$. In fact, every element $x \in M
  \tensor_R R((z))$ can be written (in at least one way) in the form
  $x = \sum_{j=1}^s m_j \tensor f_j$, with $m_j \in M$ and $f_j =
  \sum_{i \geq k} r_{ij} z^i \in R((z))$, and $\Phi_M(x) = \sum_{i
    \geq k} \big( \sum_{j=1}^s m_j r_{ij} \big) z^i$; commutativity
  of~\eqref{diag:chase} shows that this is well defined. By the Five
  Lemma, the map~$\Phi_M$ is an isomorphism in general as claimed.
\end{proof}

\begin{remark}
  The lemma fails for modules which are not finitely
  generated. Specifically, if $M$~is free of infinite rank one can
  still define a map $M \tensor_R R((z)) \rTo M((z))$, essentially in
  the same way as above, and linear independence of basis elements
  guarantees that this map is injective. Its image consists precisely
  of those formal \textsc{Laurent} series $\sum_{i \geq k} m_i z^i$
  which have the property that the submodule of~$M$ generated by the
  set of coefficients $\{m_i \,|\, i \geq k\}$ is finitely
  generated. Using this and a diagram chase in~\eqref{diag:chase} one
  can show that $\Phi_M$ is surjective whenever $M$~is finitely
  generated; in that case $\Phi_M$~is injective as well if and only if
  $M$~is finitely presented.
\end{remark}

\begin{definition}
  \label{def:novikov_cohomology}
  Let $B$ be a cochain complex of $R[z,z^{-1}]$\nbd-modules. The {\it
    positive \textsc{Novikov} cohomology\/} is the cohomology of the
  cochain complex $B \tensor_{R[z,z^{-1}]} R((z))$. The {\it negative
    \textsc{Novikov} cohomology\/} is the cohomology of the cochain
  complex $B \tensor_{R[z,z^{-1}]} R((z^{-1}))$.
\end{definition}

\begin{definition}
  Let $C$ be a cochain complex of right $R$\nbd-modules, and let $h
  \colon C \rTo C$ be a cochain map. The {\it mapping torus $T(h)$
    of~$h$} is defined by
  \[T(h) = \mathrm{Cone}\, \big( C \tensor_R R[z,z^{-1}]
  \rTo[l>=5em]^{h \tensor \id - \id \tensor z} C \tensor_R R[z,z^{-1}]
  \big)\] where the map ``$z$'' denotes the self map of $R[z,z^{-1}]$
  given by multiplication by the indeterminate~$z$.
\end{definition}

In this definition ``Cone'' stands for the algebraic mapping cone; if
a map of cochain complexes $f \colon X \rTo Y$ is considered as a
double complex~$D^{*,*}$ concentrated in columns $p=-1,0$ with
horizontal differential~$f$, and differential of~$X$ changed by a sign
$-1$, then $\mathrm{Cone}\,(f) = \mathrm{Tot}_\oplus
D^{*,*}$. Explicitly, we have $\mathrm{Cone}\, (f)^n = X^{n+1} \oplus
Y^n$, and the differential is given by the following formula:
\begin{align*}
  \mathrm{Cone}\,(f)^n = X^{n+1} \oplus Y^n &\rTo X^{n+2} \oplus
  Y^{n+1} = \mathrm{Cone}\,(f)^{n+1} \\
  (x,\,y) & \mapsto \quad \big(-d (x),\, f(x) + d (y)\big)
\end{align*}

\begin{theorem}
  \label{thm:torus_cohomology}
  Let $C$ be a (possibly unbounded) cochain complex of finitely
  presented right $R$\nbd-modules, and let $h \colon C \rTo C$ be an
  arbitrary cochain map. Then the negative \textsc{Novikov\/}
  cohomology of the mapping torus~$T(h)$ of~$h$ is trivial, \ie, the
  cochain complex $T(h) \tensor_{R[z,z^{-1}]} R((z^{-1}))$ is
  acyclic. --- If $h$ is a quasi-isomorphism, then the positive
  \textsc{Novikov} homology of~$T(h)$ is trivial as well, \ie, the
  cochain complex $T(h) \tensor_{R[z,z^{-1}]} R((z))$ is acyclic in this
  case.
\end{theorem}

\begin{proof}
  We deal with negative \textsc{Novikov} cohomology first. Since
  tensor products are additive, we have an equality of cochain
  complexes
  \[T(h) \tensor_{R[z,z^{-1}]} R((z^{-1})) = \mathrm{Cone}\, \big( C
  \tensor_R R((z^{-1})) \rTo[l>=5em]^{h \tensor \id - \id \tensor z} C
  \tensor_R R((z^{-1})) \big) \ .\] Using Lemma~\ref{lem:series} we
  identify the complex $C \tensor_R R((z^{-1}))$ with $C((z^{-1}))$.
  We can now write
  \begin{equation}
    \label{eq:identification}
    T(h) \tensor_{R[z,z^{-1}]} R((z^{-1})) = \totrt (D^{*,*})
    \end{equation}
    where~$D^{*,*}$ is defined as follows:
  \begin{align*}
    D^{p,q} &= C^{p+q+1} \oplus C^{p+q} \\
    d^h &\colon \ \, D^{p,q} \rTo D^{p+1,q} \\
    & \hphantom{\colon\ \,} (x,y) \mapsto (0, -x) \\
    d^v & \colon \ \, D^{p,q} \rTo D^{p,q+1} \\
    & \hphantom{\colon \ \,} (x,y) \mapsto \big( -d^C(x),\,
    h(x)+d^C(y) \big)
  \end{align*}
  Here $d^C$ denotes the coboundary map in the complex~$C$. We have
  $d^h \circ d^h = 0$, and the $p$th column $D^{p,*}$ of~$D^{*,*}$ is
  the $p$th shift of~$\mathrm{Cone}\, (h)$ so that $d^v \circ d^v =
  0$. Finally, the differentials anti-commute: for a typical element
  $(x,y) \in D^{p,q} = C^{p+q+1} \oplus C^{p+q}$ we have
  \begin{multline*}
    d^v \circ d^h (x,y) = d^v (0, -x) = \big( 0, d^C(-x)
    \big) = -\big(0, d^C(x)\big) \\ = -d^h \big( (-d^C(x),
    h(x)+d^C(y)) \big) = -d^h \circ d^v (x,y) \ .
  \end{multline*}
  To complete the identification given in~\eqref{eq:identification} we
  note that the $p$th column of~$D^{*,*}$ corresponds to the terms
  with coefficient $z^p$ in the formal \textsc{Laurent} series
  notation. --- Now the rows of $D^{*,*}$ are clearly exact so
  that $\totrt (D^{*,*})$ is acyclic by Proposition~\ref{prop:main}.

  For positive \textsc{Novikov} cohomology we note that since $C
  \tensor_R R((z)) = C((z))$ we can identify $T(h)
  \tensor_{R[z,z^{-1}]} R((z))$ with $\totlt (D^{*,*})$. If $h$ is a
  quasi-isomorphism then $\mathrm{Cone}\,(h)$ is acyclic so that the
  columns of~$D^{*,*}$ are exact. By Proposition~\ref{prop:main},
  $\totlt(D^{*,*})$ is acyclic.
\end{proof}

\begin{remark}
  Suppose $h$ is the map $\bZ \rTo \bZ$ given by multiplication
  by~$2$, considered as a cochain map concentrated in cochain
  degree~$0$. Then the~$D^{*,*}$ in the proof above is the double
  complex described in Remark~\ref{rmk:not-acyclic} which has the
  property that $\totlt (D^{*,*})$ is not acyclic. This provides an
  example of a map~$h$ whose mapping torus has trivial negative but
  non-trivial positive \textsc{Novikov\/} cohomology.

  The asymmetry stems from the fact that the definition of mapping
  tori involves a choice. One could have defined the mapping torus
  of~$h$ as the mapping cone of $h \tensor \id - \id \tensor z^{-1}$
  in which case the roles of positive and negative \textsc{Novikov}
  cohomology in Theorem~\ref{thm:torus_cohomology} are reversed. This
  can be shown by identifying the $p$th column of~$D^{*,*}$ in the
  proof above with the coefficients of~$z^{-p}$ in the
  \textsc{Laurent} series notation, or by using double complexes with
  differentials going down and left (in which case the roles of
  $\totlt$ and $\totrt$ are swapped in Proposition~\ref{prop:main}).
\end{remark}

\begin{corollary}
  \label{cor:ranicki}
  Suppose that $C$ is a bounded above cochain complex of projective
  right $R[z,z^{-1}]$\nbd-modules. Suppose further that $C$ is
  homotopy equivalent, as an $R$\nbd-module complex, to a bounded
  complex~$B$ of finitely generated projective right
  $R$\nbd-modules. Then $C$ has trivial positive and negative
  \textsc{Novikov} cohomology, that is, the two cochain complexes $C
  \tensor_{R[z,z^{-1}]} R((z))$ and $C \tensor_{R[z,z^{-1}]}
  R((z^{-1}))$ are acyclic.
\end{corollary}

\begin{proof}
  Let $f \colon C \rTo B$ and $g \colon B \rTo C$ mutually inverse
  $R$\nbd-linear homotopy equivalences. There are
  $R[z,z^{-1}]$\nbd-linear homotopy equivalences
  \[C \lTo T(zgf) \rTo T(fzg)\] where ``$z$'' denotes the self map
  given by multiplication by~$z$; a proof can be found, for example,
  in \cite[\S \S2--3]{Iteration}. It follows that the \textsc{Novikov}
  cohomology of~$C$ and of~$T(fzg)$ are the same. Now $fzg$ is a
  homotopy equivalence as $z$~acts invertibly on~$C$, and
  Theorem~\ref{thm:torus_cohomology} assures us that $T(fzg)$ has
  trivial positive and negative \textsc{Novikov} cohomology.
\end{proof}

This Corollary is the ``only-if'' part of a result obtained by
\textsc{Ranicki} \cite[Theorem~2]{Ranicki-findom} using different
methods; the present proof has the advantage of being completely
elementary.

\begin{remark}
  \label{rem:T_by_double}
  With $D^{*,*}$ as in the proof of Theorem~\ref{thm:torus_cohomology}
  we can identify $T(h)$ with $\mathrm{Tot}_\oplus D^{*,*}$. Note that
  $D^{*,*}$ has exact rows, and has exact columns if $h$~is a
  quasi-isomorphism. In view of Remark~\ref{rmk:not-acyclic} this does
  {\it not\/} imply that $T(h)$ is acyclic.
\end{remark}

\providecommand{\bysame}{\leavevmode\hbox to3em{\hrulefill}\thinspace}
\providecommand{\href}[2]{#2}

\end{document}